\documentclass[a4paper,12pt,reqno]{amsart}
\usepackage{amsmath,amssymb}

\usepackage{graphicx}
\usepackage{amsthm}
\usepackage{thmtools}
\usepackage{comment}
\usepackage{hyperref}
\usepackage{float}
\usepackage{caption}
\usepackage{mathtools}
\usepackage{arcs}
\usepackage{yhmath}
\usepackage{tikz}
\usetikzlibrary{calc,intersections}
\makeatletter

\@namedef{subjclassname@2020}{%
\textup{2020} Mathematics Subject Classification} 
\makeatother

\newtheorem*{definition}{Definition}

\newtheorem{theorem}{Theorem}
\newtheorem{lemma}[theorem]{Lemma}

\title{On Ulam's Segment Motion Problem}

\author{Dragomir Grozev}
\address{Institute of Mathematics and Informatics, Bulgarian Academy of Sciences, Acad. G. Bonchev 8,
1113 Sofia, Bulgaria}
\email{drago.grozev@gmail.com}

\author{Nikolai Nikolov}
\address{Institute of Mathematics and Informatics, Bulgarian Academy of Sciences, Acad. G. Bonchev 8,
1113 Sofia, Bulgaria
\vspace{1mm}
\newline Faculty of Information Sciences, State University of Library Studies and Information
Technologies, Shipchenski prohod 69A, 1574 Sofia, Bulgaria}
\email{nik@math.bas.bg}

\thanks{The second named author was partially supported by the Bulgarian National Science Fund,
Ministry of Education and Science of Bulgaria under contract
KP-06-N82/6.} 

\begin{document}


\begin{abstract}

We study extremal rigid motions of a unit segment in $\mathbb{R}^d$, $d\ge 2$. Given two prescribed positions of a unit segment, we consider continuous motions transforming the initial position into the final one and investigate the total length of the trajectories traced by its endpoints. This minimization problem was posed by Ulam~\cite{Ulam1960} and solved by Gurevich~\cite{Gurevich1977} and Dubovitskii~\cite{Dubovitskii1976}.

Two natural lower bounds are given by the sum of the endpoint displacements and by the angle between the initial and final directions of the segment. We characterize all pairs of segment positions for which either of these lower bounds is attained.

In arbitrary dimension, we obtain complete characterizations of the equality cases for both the endpoint-displacement bound and the angular bound.
\end{abstract}

\maketitle

\vspace{-8mm}

\section{Introduction}

We consider continuous rigid motions in \(\mathbb{R}^d\), \(d\ge 2\); see, for example, \cite{Bo_Roth_1990}. Given two unit segments \(A_0B_0\) and \(A_1B_1\) in \(\mathbb{R}^d\), we study continuous rigid motions that transform \(A_0B_0\) into \(A_1B_1\) by mapping \(A_0\) to \(A_1\) and \(B_0\) to \(B_1\). We are particularly interested in motions that minimize the sum of the lengths of the trajectories of the endpoints. This problem was posed by Ulam~\cite{Ulam1960} and solved by Gurevich~\cite{Gurevich1977} and Dubovitskii~\cite{Dubovitskii1976}, who showed that an optimal motion exists and can be represented as a finite sequence of slidings and rotations.

\begin{definition}
Let \(X(t)\), \(t\in[0,1]\), denote the trajectory of a point \(X\) under a given rigid motion, where \(X(t)\in\mathbb{R}^d\). The length of this trajectory is denoted by \(L(X)\). In particular, if \(A(t)\) and \(B(t)\) are the trajectories of the endpoints of a unit segment, we write
\[
u(t):=B(t)-A(t).
\]
Since the motion is rigid, \(u(t)\) is a unit vector for every \(t\in[0,1]\). The sum of the lengths of the trajectories of the endpoints \(A_0\) and \(B_0\) under a rigid motion mapping the segment \(A_0B_0\) to \(A_1B_1\) is denoted by \(L(A_0B_0,A_1B_1)\), namely,
\[
L(A_0B_0,A_1B_1):=L(A_0)+L(B_0).
\]
\end{definition}

We restrict attention to motions satisfying \(L(A_0B_0,A_1B_1)<\infty\). Any such motion admits a reparametrization for which both trajectories \(A(t)\) and \(B(t)\) become Lipschitz curves, with a Lipschitz constant depending only on \(L(A_0B_0,A_1B_1)\); see, e.g.~\cite{Dubovitskii1976}. Accordingly, throughout the paper we assume that
\(A(t)\), \(B(t)\), and hence \(u(t)\), are Lipschitz. 

We also reserve the letter $\theta$ to denote the signed angle between $\overrightarrow{A_0B_0}$ and $\overrightarrow{A_1B_1}$, chosen in the interval $(-\pi,\pi]$.

The main problem is to determine the minimum possible value of $L(A_0B_0, A_1B_1)$ among all continuous rigid motions connecting the prescribed segment positions. Clearly, 
\begin{equation}
\label{eq:proposition_1}
L(A_0B_0, A_1B_1)\ge |A_0A_1|+|B_0B_1|.
\end{equation}
The following theorem characterizes the cases in which equality holds in~\eqref{eq:proposition_1}.
\begin{theorem}
\label{thm:iff_sliding}
Let $A_0B_0$ and $A_1B_1$ be two unit segments in $\mathbb{R}^d$. A necessary and sufficient condition for the existence of a continuous motion satisfying
$$L(A_0B_0, A_1B_1)= |A_0A_1|+|B_0B_1|$$
is
\begin{equation}
\label{eq:prop1_eq2}
\left(\overrightarrow{A_0A_1}\cdot \overrightarrow{A_0B_0}\right)
\left(\overrightarrow{A_0A_1}\cdot \overrightarrow{A_1B_1}\right)\ge 0
\end{equation}
and
\begin{equation}
\label{eq:prop1_eq3}
\left(\overrightarrow{B_0B_1}\cdot \overrightarrow{A_0B_0}\right)
\left(\overrightarrow{B_0B_1}\cdot \overrightarrow{A_1B_1}\right)\ge 0.
\end{equation}

\end{theorem}

\smallskip
Another natural lower bound arises from the angle between the directions of the segments. If $\theta$ denotes the signed angle between the directions of $A_0B_0$ and $A_1B_1$, $0\le |\theta|\le \pi$, then 
\[
L(A_0B_0,A_1B_1)\ge |\theta|.
\]

In dimension three and higher, it is natural to ask to what extent an extremal angular motion must be planar. The following estimate provides a quantitative answer by giving a lower bound for the minimum trajectory length in terms of the angle between the segment directions and the distance between the lines containing the segments.

In particular, it shows that any deviation from planarity forces the minimum trajectory length to exceed the angular lower bound.

\begin{theorem}
    \label{thm:theorem_2}
Let $A_0B_0$ and $A_1B_1$ be unit segments in $\mathbb{R}^d$, $d\ge 2$,
and let $h\ge 0$ be the distance between the lines containing them. 
Assume that the angle between their directions is $\theta, 0<|\theta|\le \pi$. Then
\[
L(A_0B_0,A_1B_1)
\ge
\sqrt{
\theta^2+ 4
h^2\cos^2\frac{\theta}{2}
}.
\]
\end{theorem}

The final section is devoted to a strengthened lower-bound estimate (Theorem~\ref{prop:self_rotating_case_in_R3}), from which Theorem~\ref{thm:theorem_2} is derived.

\smallskip
\noindent \textbf{Remark.} Theorem~\ref{thm:theorem_2} also implies that if $L(A_0B_0,A_1B_1)=|\theta|$, then the points $A_0$, $B_0$, $A_1$, and $B_1$ are coplanar. 

Denote the lower bound in Theorem~\ref{thm:theorem_2} by $G(\theta,h)$. Then $G(0,h)=2h$, $G(\theta,0)=|\theta|$, $G(\pi,h)=\pi$. The estimate is sharp when $h=0$ or $\theta=0$. In the case $\theta=\pi$ it is sharp if and only if $0\le h\le 1$; this follows from Theorem~\ref{thm:rotation_prop_1}. 

For each fixed $\theta\in (-\pi,\pi)$, the function $G(\theta, h)$ is strictly increasing in $h\ge 0$. For each fixed $h\in [0,1]$, the function $G(\theta,h)$ is strictly increasing in $\theta\in[0,\pi]$. If $h>1$, then $G(\theta, h)$ is not monotone as a function of $\theta$ on $[0,\pi]$.

\begin{theorem}
\label{thm:rotation_prop_1}
Let $A_0B_0$ and $A_1B_1$ be unit segments in $\mathbb{R}^d$, where $d\ge 2$, and let
\[
\angle(\overrightarrow{A_0B_0},\overrightarrow{A_1B_1})=\theta,
\qquad 0\le |\theta|\le \pi.
\] 
Then there exists a continuous rigid motion satisfying $L(A_0B_0,A_1B_1)=|\theta|$
if and only if the points $A_0$, $B_0$, $A_1$, and $B_1$ are coplanar and
\begin{equation}
\label{eq:rotating_1}
|A_0B_1|\le 1;\quad |A_1B_0|\le 1.
\end{equation}
\end{theorem}

\smallskip
\noindent \textbf{Remark.} For planar motions, the latter characterization is due to Icking, Rote, Welzl, and Yap~\cite{IckingRoteWelzlYap1993}. In their treatment, however, the necessity part of Theorem~\ref{thm:rotation_prop_1} appears only implicitly as a consequence of their main theorem, whereas our proof is direct and follows a substantially different approach.

The sufficiency part was also established in~\cite{IckingRoteWelzlYap1993}, where it was shown that condition~\eqref{eq:rotating_1} guarantees the existence of an extremal angular motion consisting of at most three rotations about the endpoints of the segment. Our proof shows that two rotations always suffice, although one of the centers of rotation may lie in the interior of the segment. The special case in which a single rotation suffices is discussed after the proof of the theorem.

\smallskip
The main objective of this paper is to prove Theorems~\ref{thm:iff_sliding}, \ref{thm:theorem_2}, \ref{thm:rotation_prop_1}, and Theorem~\ref{prop:self_rotating_case_in_R3}, and to characterize all cases in which equality holds in Theorems~\ref{thm:iff_sliding} and~\ref{thm:rotation_prop_1}.

\section{Proof of Theorem~\ref{thm:iff_sliding}}
Intuitively, conditions~\eqref{eq:prop1_eq2} and \eqref{eq:prop1_eq3}
express the fact that the projections of the moving segment onto the two
sliding directions do not change sign during the motion.

We may assume that $d=3$. The case $A_0B_0\equiv A_1B_1$ is trivial. It is clear that when either $A_0=A_1$ or $B_0=B_1$ then the equality in \eqref{eq:proposition_1} is impossible and either \eqref{eq:prop1_eq2} or \eqref{eq:prop1_eq3} is not satisfied. Let $\ell_A$ and $\ell_B$ be the two lines determined by $A_0A_1$ and $B_0B_1$, respectively. The case $\alpha=0$, that is, when $\ell_A \parallel \ell_B$, can be checked straightforwardly. We further assume $\ell_A\nparallel \ell_B$. Denote by \(h\) the distance between these lines and let \(\alpha\) be the angle between the lines \(\ell_A\) and \(\ell_B\). Since the length of the trajectory of \(A\) is at least \(|A_0A_1|\), and the length of the trajectory of \(B\) is at least \(|B_0B_1|\), we obtain
\[
L(A_0B_0,A_1B_1)\ge |A_0A_1|+|B_0B_1|.
\]
Equality holds precisely when both trajectories are shortest paths,
namely when \(A\) moves monotonically along \(A_0A_1\) and \(B\) moves
monotonically along \(B_0B_1\).

Hence it remains to determine when a unit segment can slide continuously
from \(A_0B_0\) to \(A_1B_1\) while its endpoints move monotonically
along the segments \(A_0A_1\) and \(B_0B_1\), respectively.

Let $P$ be a plane parallel to the lines $\ell_A$ and $\ell_B$. Let $\ell'_A$ and $\ell'_B$ be the projections of the lines $\ell_A$ and $\ell_B$ onto $P$ respectively and \(O\) be the intersection point of $\ell'_A$ and $\ell'_B$. Since the moving segment has length \(1\), it follows that \(h<1\).

Consider the circle
\[
k=\left\{t\in P:\ |Ot|=\frac{\sqrt{1-h^2}}{\sin\alpha}\right\}.
\]
\begin{figure}[H]
    \centering
    \includegraphics[width=0.5\linewidth]{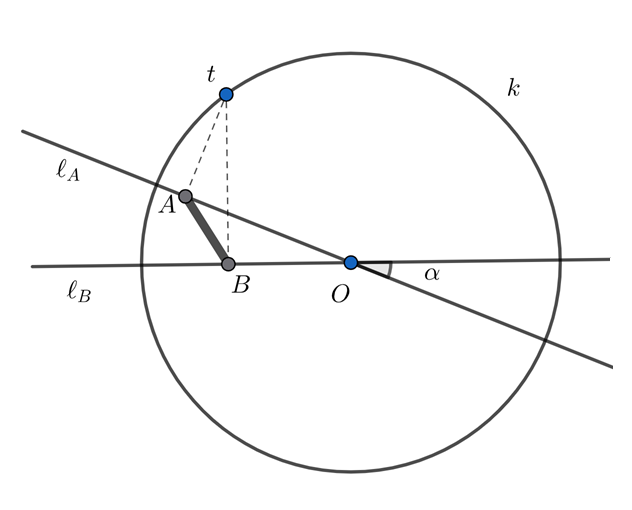}
    \caption{}
    \label{fig:theorem_1}
\end{figure}

For every point \(t\in k\), let \(A_t\) and \(B_t\) be the orthogonal projections of \(t\) onto \(\ell_A\) and \(\ell_B\), respectively. (Fig.~\ref{fig:theorem_1} illustrates the planar case.) A straightforward computation shows that
\[
|A_tB_t|=1.
\]
Conversely, every position of a unit segment with endpoints on \(\ell_A\) and \(\ell_B\) arises in this way. Hence the admissible configurations of the moving segment are parametrized by the points of the circle \(k\).

The lines $\ell'_A,\ell'_B$ divide \(k\) into four arcs. On each such arc, the points \(A_t\) and \(B_t\) move monotonically along \(\ell_A\) and \(\ell_B\), respectively. Indeed, along a fixed arc the signs of the corresponding projection coordinates remain unchanged, and therefore neither endpoint can reverse its direction of motion.

Consequently, a length-minimizing sliding motion exists if and only if the points of \(k\) corresponding to the initial and final configurations belong to the same closed arc of this decomposition. This is equivalent to requiring that the projections of the initial and final unit segments onto each sliding direction have the same sign.

For the direction \(\ell_A\), this condition is
\[
\left(\overrightarrow{A_0A_1}\cdot \overrightarrow{A_0B_0}\right)
\left(\overrightarrow{A_0A_1}\cdot \overrightarrow{A_1B_1}\right)\ge 0.
\]
For the direction \(\ell_B\), it is
\[
\left(\overrightarrow{B_0B_1}\cdot \overrightarrow{A_0B_0}\right)
\left(\overrightarrow{B_0B_1}\cdot \overrightarrow{A_1B_1}\right)\ge 0.
\]
These are precisely conditions \eqref{eq:prop1_eq2} and \eqref{eq:prop1_eq3}.
If the two inequalities hold, the corresponding points of \(k\) can be joined by an arc contained in one of the four sectors. Along this arc, both \(A_t\) and \(B_t\) move monotonically from their initial to their final positions. Therefore $L(A_0B_0,A_1B_1)=|A_0A_1|+|B_0B_1|$.

Conversely, if $L(A_0B_0,A_1B_1)=|A_0A_1|+|B_0B_1|$ then both endpoints must move monotonically along the segments \(A_0A_1\) and \(B_0B_1\). The corresponding point \(t\in k\) therefore remains in a single sector of the above decomposition. Consequently, the projections of the initial and final unit segments onto each sliding direction have the same sign, which yields conditions \eqref{eq:prop1_eq2} and \eqref{eq:prop1_eq3}.

\smallskip
\noindent \textbf{Remark.} It can be shown that the line through \(O\) perpendicular to $P$ is the instantaneous axis of rotation of the moving segment $A_tB_t$; see, for example, \cite{Bo_Roth_1990}.

\section{Proof of Theorem~\ref{thm:rotation_prop_1}}

\noindent \textbf{Necessity.}

Theorem~\ref{thm:theorem_2} shows that if $L(A_0B_0,A_1B_1)=|\theta|$ then $A_0,B_0,A_1,B_1$ are coplanar and the extremal motion remains in $\mathbb{R}^2$.

\begin{definition}
A planar motion of a segment is called an \textbf{extremal angular motion} if
\[
L(A_0B_0,A_1B_1)=|\theta|.
\]
\end{definition}
\noindent 
\textbf{Remark.} A basic example of an extremal angular motion is a rotation through an angle
\(\theta\), with \(|\theta|\le\pi\), about a point \(O\in A_0B_0\).
More generally, it may consist of a sequence of rotations in the same direction,
whose total angle is \(\theta\), where \(|\theta|\le\pi\).
There are many other extremal angular motions. The following result shows that,
for any such motion, the instantaneous center of rotation must lie on the
segment \(AB\) at every moment.

\smallskip
Suppose we have an extremal angular motion and let $A(t), B(t),t\in[0,1]$ be the curves traced by the points $A$ and $B$, respectively, where $A(0)=A_0, A(1)=A_1$, $B(0)=B_0, B(1)=B_1$. Set $u(t):=B(t)-A(t)$. The angle between $u(0)$ and $u(1)$ is $\theta$. 

It suffices to prove that $|B(1)-A(0)|\le 1$. The other condition follows by symmetry. We have
\begin{align*}
    L(A_0B_0, A_1B_1) &= \int_0^1|\dot{A}(t)|+|\dot{B}(t)|\,dt \\
    &\ge \int_0^1 |\dot{B}(t)-\dot{A}(t)|\ dt = \int_0^1 |\dot{u}(t)|\ dt \\
    &\ge |\theta|  
\end{align*}
The equality is attained if and only if: 
\begin{enumerate}
 \item 
 $|\dot{A}(t)|+|\dot{B}(t)|=|\dot{B}(t)-\dot{A}(t)|$ holds almost everywhere, that is, when $\dot{A}(t)$ and $\dot{B}(t)$ have opposite directions a.e. This can be written as $\dot{B}(t)= k(t)\dot{u}(t)$, $\dot{A}(t)=-(1-k(t))\dot{u}(t)$, where $u(t)=B(t)-A(t)$ and $\displaystyle k(t):=|\dot{B}(t)|/|\dot{u}(t)|$, when $\left| \dot{u}(t)\right|\ne 0$ and $k(t):=0$, when $\left| \dot{u}(t)\right|=0$. It follows that $k(t)$ is measurable and satisfies $0\le k(t)\le 1$. \\
\item  
$u(t)$ must rotate monotonically. 
\end{enumerate}

Without loss of generality, we may assume that $0\le \theta\le \pi$, $A(0)=(0,0)$, $B(0)=(1,0)$. Since \(u(t)\) rotates monotonically, the angle \(\varphi(t)\) between \(u(t)\) and \(u(0)\) is a monotone function. We may therefore use \(\varphi\) as a parameter of the motion and write
\[
B(\theta)= B(0)+\int_0^{\theta} k(\varphi) u'(\varphi)\, d\varphi.
\]
Let us denote $v:=B(\theta)-A(0)=B(\theta)$. Since $u'(\varphi)=(-\sin\varphi, \cos\varphi)$ and $B(0)=(1,0)$, we obtain
\[
v_x=1-\int_0^\theta k(\varphi)\sin\varphi\,d\varphi,
\quad
v_y=\int_0^\theta k(\varphi)\cos\varphi\,d\varphi,
\]
where \(k:[0,\theta]\to[0,1]\) is a measurable function. We can set $k(\varphi)=0, \varphi\in (\theta,\pi]$ and write 
\[
v_x=1-\int_0^\pi k(\varphi)\sin\varphi\,d\varphi=\int_0^{\pi} \left(\frac{1}{2}-k(\varphi)\right)\sin\varphi \,d\varphi,
\quad
v_y=\int_0^\pi k(\varphi)\cos\varphi\,d\varphi
\]
Let $k_1(\varphi):=\frac{1}{2}-k(\varphi)$, $k_1(\varphi)\in [-1/2,1/2]$, when $\varphi\in[0,\pi]$. Using $\int_0^{\pi}\cos\varphi\,d\varphi=0$, we obtain
\begin{equation}
\label{eq:parseval}
    v_x=\int_0^{\pi} k_1(\varphi)\sin\varphi \,d\varphi, \quad v_y=-\int_0^\pi k_1(\varphi)\cos\varphi\,d\varphi
\end{equation}
 We will prove that \(|v|\le 1\). Let \(p=(\cos\alpha,\sin\alpha)\) be an arbitrary unit vector. Then
\begin{align*}   
p\cdot v &=\int_0^\pi k_1(\varphi)\left(\sin\varphi\cos\alpha -\cos\varphi\sin\alpha\right) \,d\varphi\\
&=\int_0^{\pi} k_1(\varphi) \sin(\varphi-\alpha)\, d\varphi
\end{align*}
Thus,
\[
|p\cdot v|\le \frac{1}{2}\int_0^{\pi} \left|\sin(\varphi-\alpha)\right|\, d\varphi= \frac{1}{2}\int_0^{\pi} \left|\sin(\varphi)\right|\, d\varphi =1,
\]
since $|k_1(\varphi)|\le 1/2$ and $|\sin\varphi|$ is $\pi$-periodic. This inequality holds for every unit vector \(p\) and we conclude $|v|\le 1$.

\smallskip
\noindent \textbf{Remark.} Observe that \( |v|\le 1 \) can also be deduced from \eqref{eq:parseval} via Parseval's identity. Indeed, \(k_1\in L^2([0,2\pi])\) and \(\|k_1\|_{L^2}\le \sqrt{2\pi}\).

\smallskip
\noindent\textbf{Sufficiency.} Suppose that condition \eqref{eq:rotating_1} holds. We shall prove that the segment $A_0B_0$ can be moved to $A_1B_1$ by means of two rotations whose total angle in absolute value does not exceed $\pi$. Observe that there is a symmetry between the two segments, that is, if $A_1B_1$ can be moved to $A_0B_0$ by two rotations, then by applying the inverse rotations, $A_0B_0$ can also be moved to $A_1B_1$ by two rotations.

Consider a coordinate system in which $\{A_0,A_1\}=\{(0,0), (a,0)\}$. By the symmetry described above, we may assume that $A_0=(0,0)$. Let $M$ be the midpoint of $A_0A_1$, and let $\ell$ be the perpendicular bisector of $A_0A_1$. The assumptions imply that $a\le 2$ and that $\ell$ intersects the segments $A_0B_0$ and $A_1B_1$ at points $X_0$ and $X_1$, respectively. There are two possibilities: 1) $X_0$ and $X_1$ lie in the same half-plane with respect to $Ox$; 2) $X_0$ and $X_1$ lie in different half-planes.

\smallskip
\noindent {\bf Case 1.} $X_0$ and $X_1$ lie in the same half-plane with respect to $Ox$, say the upper half-plane. Without loss of generality, we may assume that $X_0M\le X_1M$, that is, $\angle X_0A_0A_1\le \angle X_1A_1A_0$. Indeed, if $|X_0M|\ge |X_1M|$, we consider the reflection with respect to $\ell$, which interchanges the roles of $A_0B_0$ and $A_1B_1$.

Under this assumption, rotate $A_0B_0$ about the point $A_0(0,0)$ in the positive direction until it passes through the point $X_1$ (Fig.~\ref{fig:rotation_sufficient_1}, left side). This is possible because $|A_0X_1|=|A_1X_1|\le 1$. Let $B_0'$ denote the image of $B_0$ at that moment. Next, rotate $A_0B_0'$ in the same (positive) direction about the point $X_1$ until $A_0$ reaches the point $A_1$. We shall prove that the total angle of rotation $\theta$ satisfies $0<\theta\le \pi$.

\begin{figure}[H]
    \centering
    \includegraphics[width=0.6\linewidth]{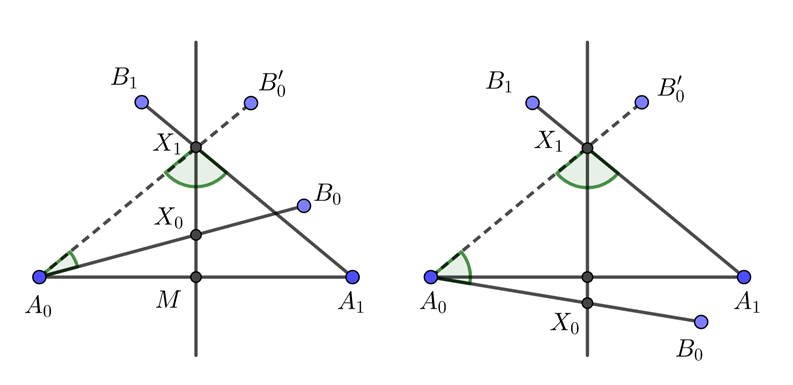}
     \caption{}
    \label{fig:rotation_sufficient_1}     
\end{figure}

The first rotation is through the angle $\angle X_1A_1A_0- \angle X_0A_0A_1\ge 0$. The second rotation is through the angle $\pi-2 \angle X_1A_1A_0$. Hence,
\[(\angle X_1A_1A_0- \angle X_0A_0A_1)+(\pi-2\angle X_1A_1A_0)=\pi-\angle X_0A_0A_1 - \angle X_1A_0A_1\le \pi.\]

\smallskip
\noindent {\bf Case 2.} $X_0$ and $X_1$ lie in different half-planes with respect to $Ox$. Assume for definiteness that $X_0$ lies in the upper half-plane (Fig.~\ref{fig:rotation_sufficient_1}, right side). By symmetry, we may assume without loss of generality that $|X_0M|\le |X_1M|$, that is, $\angle X_0A_0A_1\le \angle X_1A_1A_0$.

Rotate $A_0B_0$ in the negative direction about $A_0$ until it passes through the point $X_1$. The point $B_0$ is moved to a point $B'_0$. Then rotate $A_0B'_0$ in the same (negative) direction about the point $X_1$ so that $A_0$ is moved to $A_1$. The total angle of rotation is
\[
( \angle X_0A_0A_1+\angle X_1A_1A_0)+ (\pi-2\angle X_1A_1A_0)=\pi +\angle X_0A_0A_1-\angle X_1A_1A_0\le \pi. 
\]
\smallskip
\noindent \textbf{Remark.} In the proof of sufficiency we established that, whenever the conditions of Theorem~\ref{thm:rotation_prop_1} are satisfied, the segment $A_0B_0$ can be moved to $A_1B_1$ by means of two rotations in the same direction whose total angle is $\theta$, $0<|\theta|\le \pi$. According to the above argument, one of these rotations is about the point $A_0$, while the other is about the point $X_1$, although their order may be reversed. In the proof presented above, the first rotation is about $A_0$ whenever $X_0M\le X_1M$. This condition is equivalent to $A_1B_0\le A_0B_1$.

In the case when $|X_0M|> |X_1M|$, that is, $|A_1B_0|> |A_0B_1|$, the first rotation is about a different point. However, if we interchange the roles of $A_i$ and $B_i$, then from $|B_1A_0|<|B_0A_1|$ it follows that we may first rotate about $B_0$ and then about another point. Thus, the first rotation can always be chosen to be either about $A_0$ or about $B_0$.

It is interesting to note that the extremal motion can be realized by a single rotation only when the segments $A_0B_0$ and $A_1B_1$ intersect at a point $O$ satisfying $|OA_0|=|OA_1|$. Indeed, there exists a unique rotation centered at $O$ that maps $A_0B_0$ to $A_1B_1$. In this case,
\[
L(A_0B_0,A_1B_1)=|OA_0||\theta|+|OB_0||\theta|,
\]
and in order to have $L(A_0B_0,A_1B_1)=|\theta|$, it is necessary that $|OA_0|+|OB_0|=1=|A_0B_0|$, that is, $O\in A_0B_0$. In the same way, $O\in A_1B_1$, thus $O\in A_0B_0\cap A_1B_1$.

\section{A sharper estimate and the proof of Theorem~\ref{thm:theorem_2}}

The proof of Theorem~\ref{thm:theorem_2} is based on the following strengthened lower-bound theorem.
\begin{theorem}
\label{prop:self_rotating_case_in_R3}
Let $A_0B_0$ and $A_1B_1$ be unit segments in $\mathbb{R}^d$, $d\ge 3$,
and let $h\ge 0$ be the distance between the lines containing them. 
Assume that the angle between their directions is $\theta, 0<|\theta|\le \pi$. Then
\[
L(A_0B_0,A_1B_1)
\ge
F(\theta, h),
\]
where
\[
F(\theta,h)=
\min_{0\le z\le 1}\,\max\left\{
\sqrt{4\left(\max\{h,z\}\right)^2+\ell(z)^2},
\,2\arccos\left(\sqrt{1-z^2}\cos\frac{\theta}{2}\right)
\right\}.
\]

\[
\ell(z)=
\begin{cases}
2z+\sqrt{1-z^2}\bigl(|\theta|-2\arcsin z\bigr),
&
0\le z\le \sin{\frac{|\theta|}{2}}
\\[8pt]
2\sin\dfrac{|\theta|}{2},
&
\sin{\frac{|\theta|}{2}}<z\le 1
\end{cases}
\]
\end{theorem}

The proof of Theorem~\ref{prop:self_rotating_case_in_R3} is deferred until the end of this section.

\smallskip
\noindent \textbf{Proof of Theorem}~\ref{thm:theorem_2}. 

We may assume that \(0\le \theta\le \pi\). 
We prove that
\[
F(\theta,h)\ge \sqrt{\theta^2+4h^2\cos^2\frac{\theta}{2}}.
\]

Let
\[
B(z)=2\arccos\left(\sqrt{1-z^2}\cos\frac{\theta}{2}\right),
\]
thus
\begin{equation}
\label{eq:F_theta_h}
F(\theta,h)=
\min_{0\le z\le 1}\,\max\left\{
\sqrt{4\left(\max\{h,z\}\right)^2+\ell(z)^2},
\,B(z)
\right\}.
\end{equation}

Fix \(z\in[0,1]\) and set
\[
R=\sqrt{\theta^2+4h^2\cos^2\frac{\theta}{2}}.
\]
Suppose, for contradiction, that both terms inside the maximum in~\eqref{eq:F_theta_h} are strictly smaller than \(R\). Then
\[
4h^2+\ell(z)^2 \le 4\left(\max\{h,z\}\right)^2+\ell(z)^2<R^2
\]
and
\[
B(z)^2<R^2.
\]
Multiplying the first inequality by \(\cos^2(\theta/2)\), the second by
\(\sin^2(\theta/2)\), and adding, gives
\[
\cos^2\frac{\theta}{2}\,\ell(z)^2
+
\sin^2\frac{\theta}{2}\,B(z)^2
<
R^2
-
4h^2\cos^2\frac{\theta}{2}=\theta^2.
\]
Thus
\begin{equation}
\label{eq:auxiliary_ineq_opposite}
\cos^2\frac{\theta}{2}\,\ell(z)^2
+
\sin^2\frac{\theta}{2}\,B(z)^2
<
\theta^2,
\end{equation}
We shall prove that actually the reverse inequality holds, namely
\begin{equation}
\label{eq:auxiliary_ineq}
H(z):=\cos^2\frac{\theta}{2}\,\ell(z)^2
+
\sin^2\frac{\theta}{2}\,B(z)^2
\ge \theta^2.
\end{equation}
This will yield the desired contradiction. Indeed, we will show that \(H(z)\ge H(0)=\theta^2\).

First suppose that
\[
0\le z\le \sin\frac{\theta}{2}.
\]
Then
\[
\ell(z)=2z+\sqrt{1-z^2}\left(\theta-2\arcsin z\right),
\]
and hence
\[
\ell'(z)
=
-\frac{z}{\sqrt{1-z^2}}\left(\theta-2\arcsin z\right)\le 0.
\]
Thus \(\ell(z)\le \ell(0)=\theta\). Moreover,
\[
B'(z)
=
\frac{2z\cos(\theta/2)}
{\sqrt{1-z^2}\,
\sqrt{\sin^2(\theta/2)+z^2\cos^2(\theta/2)}}.
\]
Since
\[
\frac{B(z)}{2}
=
\arccos\left(\sqrt{1-z^2}\cos\frac{\theta}{2}\right)
\ge \frac{\theta}{2},
\]
and since \(x/\sin x\) is increasing on \([0,\pi/2]\), we have
\[
\frac{B(z)}
{\sqrt{\sin^2(\theta/2)+z^2\cos^2(\theta/2)}}
=
\frac{B(z)}{\sin(B(z)/2)}
\ge
\frac{\theta}{\sin(\theta/2)}.
\]
Therefore
\[
\begin{aligned}
H'(z)
&=
2\cos^2\frac{\theta}{2}\,\ell(z)\ell'(z)
+
2\sin^2\frac{\theta}{2}\,B(z)B'(z) \\
&=
\frac{2z}{\sqrt{1-z^2}}
\left[
-\cos^2\frac{\theta}{2}\,\ell(z)\left(\theta-2\arcsin z\right)
+
2\sin^2\frac{\theta}{2}\cos\frac{\theta}{2}\,
\frac{B(z)}
{\sqrt{\sin^2(\theta/2)+z^2\cos^2(\theta/2)}}
\right] \\
&\ge
\frac{2z}{\sqrt{1-z^2}}
\left[
-\theta^2\cos^2\frac{\theta}{2}
+
2\theta\sin\frac{\theta}{2}\cos\frac{\theta}{2}
\right] \\
&=
\frac{2z\theta\cos(\theta/2)}{\sqrt{1-z^2}}
\left[
2\sin\frac{\theta}{2}-\theta\cos\frac{\theta}{2}
\right].
\end{aligned}
\]
The last expression is nonnegative, since
\[
2\sin\frac{\theta}{2}-\theta\cos\frac{\theta}{2}\ge0
\]
is equivalent to
\[
\tan\frac{\theta}{2}\ge \frac{\theta}{2}.
\]
Hence \(H\) is nondecreasing on \([0,\sin(\theta/2)]\).

Now suppose that
\[
\sin\frac{\theta}{2}\le z\le 1.
\]
Then $\displaystyle \ell(z)=2\sin\frac{\theta}{2}$ is constant, while \(B(z)\) is increasing. Hence \(H\) is nondecreasing on this interval as well. Therefore
\[
H(z)\ge H(0)=\theta^2,
\]
which proves \eqref{eq:auxiliary_ineq}, contradicting \eqref{eq:auxiliary_ineq_opposite}. Therefore, for every \(z\in[0,1]\),
\[
\max\left\{
\sqrt{4\left(\max\{h,z\}\right)^2+\ell(z)^2},
\,B(z)
\right\}
\ge R.
\]
Taking the minimum over \(z\in[0,1]\), we obtain
\[
F(\theta,h)\ge \sqrt{\theta^2+4h^2\cos^2\frac{\theta}{2}}.
\]
Combining this with Theorem~\ref{prop:self_rotating_case_in_R3} completes the proof of Theorem~\ref{thm:theorem_2}.
\qed

\medskip
\noindent \textbf{Proof of Theorem}~\ref{prop:self_rotating_case_in_R3}.

We may assume that $0\le \theta\le \pi$. Let $P_0$ and $P_1$ be the two parallel hyperplanes at distance $h$ with $A_i,B_i\in P_i$, $i=0,1$. We introduce a coordinate system such that the $Oz$-axis is perpendicular to $P_0$ and directed towards $P_1$. For any vector $v\in\mathbb{R}^d$, we denote by $v_z$ and $v_P$ its components in the direction of the $z$-axis and parallel to $P_0$, respectively. We also denote $u(t)=B(t)-A(t)$. We have
\[
\begin{aligned}
&\int_0^1 |\dot{A}(t)|+|\dot{B}(t)|\,dt\\
&=\int_0^1 |\dot{A}_P(t)+\dot{A}_z(t)e_z|+|\dot{B}_P(t)+\dot{B}_z(t)e_z|\,dt \\
&= \int_0^1 \left(|\dot{A}_P|^2+|\dot{A}_z|^2 \right)^{1/2}\,dt + \int_0^1 \left(|\dot{B}_P|^2+|\dot{B}_z|^2 \right)^{1/2}\,dt\\
&\ge \left(\left( \int_0^1|\dot{A}_P|+ |\dot{B}_P| \,dt\right)^{2} +\left(\int_0^1|\dot{A}_z|+|\dot{B}_z|\,dt\right)^2 \right)^{1/2}\\
&\ge \left(\left( \int_0^1|\dot{u}_P|\, dt \right)^{2} +\left(\int_0^1|\dot{A}_z|+|\dot{B}_z|\,dt\right)^2 \right)^{1/2}.
\end{aligned}
\]
The following Lemma is an auxiliary result. 

\begin{lemma}
\label{lem:min_path}
Let \(u_0,u_1\in\mathbb R^{d-1}\), $d\ge 3$, satisfy
\[
|u_0|=|u_1|=1,
\]
and let \(\theta\), \(0\le \theta\le \pi\), be the angle between \(u_0\) and \(u_1\).
Let \(O\) be the origin, and $0\le r\le 1$.

Among all rectifiable curves \(\gamma\subset\mathbb R^{d-1}\) joining \(u_0\) to \(u_1\)
and satisfying
\[
|x-O|\ge r
\qquad\text{for every }x\in\gamma,
\]
the minimum possible length is
\[
\ell_{\min}
=
\begin{cases}
2\sin\dfrac{\theta}{2},
&
\theta\le 2\arccos r,
\\[8pt]
2\sqrt{1-r^2}+r\bigl(\theta-2\arccos r\bigr),
&
\theta\ge 2\arccos r.
\end{cases}
\]
\end{lemma}

\begin{proof}
First, we show that the minimum is attained by a curve contained in the plane spanned by \(u_0\) and \(u_1\). 
Let \(\tilde u(\varphi)\) be the unit vector in the plane spanned by \(u_0\) and \(u_1\) such that
\[
\angle(u_0,\tilde u(\varphi))=\varphi,
\qquad
\varphi\in[0,\theta].
\]
Let \(x=x(t)\in\mathbb R^{d-1}\) be a curve avoiding the ball of radius \(r\) centered at the origin. We write
\[
x(t)=\rho(t)u(t),
\]
where \(\rho(t)\ge r\) and \(|u(t)|=1\). Reparametrizing the curve if necessary, we may assume that \(u(t)\) is parametrized by spherical arc length, that is,
\[
s=\int_0^t |\dot u(\tau)|\,d\tau,
\qquad t\in[0,1].
\]
Thus, $x(s)=\rho(s)u(s)$, $s\in[0,\ell]$, where $\rho(s)\ge r$, $u(s)\in \mathbb{R}^{d-1}$, $|u(s)|=1$ and $\displaystyle \ell:=\int_0^1 |\dot{u}|\,dt$. Consider the following planar curve
\[
\tilde{x}(s):= \rho(s)\tilde{u}\left(\frac{\theta}{\ell} s\right).
\]
It also avoids the ball $\{x:|x|<r\}$ and lies in the plane spanned by $u_0,u_1$.
We have
\[
d\tilde{x} =d \rho \tilde{u}+\rho \frac{\theta}{\ell}d\tilde{u},\quad dx =d \rho u+\rho du.
\]
Since $|u|=|\tilde{u}|=1$ and $u\perp du$, $\tilde{u}\perp d \tilde{u}$, $|u'(s)|=1$, $|\tilde{u}'(\varphi)|=1$, we get
\[
|d\tilde{x}|^2=d\rho^2|\tilde{u}|^2+\rho^2 \frac{\theta^2}{\ell^2}|d\tilde{u}|^2=d\rho^2+\rho^2\frac{\theta^2}{\ell^2}\,|ds|^2
\]
\[
|dx|^2=d\rho^2+\rho^2\,|ds|^2
\]
Since $\displaystyle \theta/\ell\le 1$, this means $|d\tilde{x}|\le |dx|$. Therefore
\[
\int_0^{\ell}|\tilde{x}'(s)|\,ds\le \int_0^{\ell}|x'(s)|\,ds 
\]

Thus, we have replaced the curve $x(t)$ by a planar one $\tilde{x}(t)$ of no greater length. Therefore, it suffices to consider the planar problem. The endpoints lie on the unit circle and have angular separation \(\theta\), while the curve is constrained to remain outside the disk of radius \(r\) centered at the origin.

If the chord joining \(u_0\) and \(u_1\) does not intersect the open disk \(|x|<r\),
then the chord is admissible and is the shortest path. The distance from the origin
to this chord is \(\cos(\theta/2)\). Hence this case occurs precisely when $\cos(\theta/2)\ge r$, or equivalently,
\[
\theta\le 2\arccos r.
\]
The minimum length is then
\[
|u_0-u_1|=2\sin\frac{\theta}{2}.
\]
Assume now that $\theta>2\arccos r$.
Then the chord intersects the forbidden disk. The shortest admissible curve consists
of a tangent segment from \(u_0\) to the circle \(|x|=r\), followed by the shorter
arc of this circle, followed by a tangent segment from the circle \(|x|=r\) to \(u_1\).

Each tangent segment has length $\sqrt{1-r^2}$. If \(\alpha\) is the angle between \(u_0\) and the radius to the tangent point, then $\cos\alpha=r$, and hence $\alpha=\arccos r$. Therefore the central angle of the circular arc is
\[
\theta-2\alpha=\theta-2\arccos r. 
\]
The length of the circular arc is $r\bigl(\theta-2\arccos r\bigr)$. Thus in this case the minimum length is
\[
2\sqrt{1-r^2}+r\bigl(\theta-2\arccos r\bigr).
\]

The two expressions coincide when \(\theta=2\arccos r\), so the proof is complete.
\end{proof}

\medskip
Now let
\[
z_m:=\max_{t\in[0,1]} |u_z(t)|.
\]
Clearly \(0\le z_m\le 1\). For every \(t\in[0,1]\),
\[
|u_P(t)|=\sqrt{1-u_z(t)^2}\ge \sqrt{1-z_m^2}.
\]
Applying Lemma~\ref{lem:min_path} with $r=\sqrt{1-z_m^2}$, we obtain $\displaystyle \int_0^1 |\dot u_P(t)|\,dt\ge \ell(z_m)$, where
\[
\ell(z_m)=
\begin{cases}
2\sin\dfrac{\theta}{2},
&
\theta\le 2\arcsin z_m,
\\[8pt]
2z_m+\sqrt{1-z_m^2}\bigl(\theta-2\arcsin z_m\bigr),
&
\theta\ge 2\arcsin z_m.
\end{cases}
\]
Indeed, this follows from Lemma~\ref{lem:min_path} because
\[
2\arccos\sqrt{1-z_m^2}=2\arcsin z_m.
\]

On the other hand
\[
\int_0^1\bigl(|\dot A_z(t)|+|\dot B_z(t)|\bigr)\,dt\ge 2\max\{h,z_m\}.
\]
Consequently,
\[
L(A_0B_0,A_1B_1)
\ge
\sqrt{4\left(\max\{h,z_m\}\right)^2+\ell(z_m)^2}.
\]

We shall also use a second estimate. Since \(u(t)\) is a curve on the unit sphere joining \(u(0)\) to \(u(1)\), and since it reaches height \(z_m\) above or below the equatorial hyperplane \(P\), its length is at least the length of the shortest spherical curve joining two equatorial points at angular distance \(\theta\) and reaching height \(z_m\). Therefore
\[
\int_0^1 |\dot u(t)|\,dt
\ge
2\arccos\left(\sqrt{1-z_m^2}\cos\frac{\theta}{2}\right).
\]
Since
\[
L(A_0B_0,A_1B_1)\ge \int_0^1 |\dot u(t)|\,dt,
\]
we obtain
\[
L(A_0B_0,A_1B_1)
\ge
2\arccos\left(\sqrt{1-z_m^2}\cos\frac{\theta}{2}\right).
\]

Thus
\[
L(A_0B_0,A_1B_1)
\ge
\max\left\{
\sqrt{4\left(\max\{h,z_m\}\right)^2+\ell(z_m)^2},
\,
2\arccos\left(\sqrt{1-z_m^2}\cos\frac{\theta}{2}\right)
\right\}.
\] 
Since $0\le z_m\le 1$, the last lower bound is at least $F(\theta, h)$, hence
\[
L(A_0B_0,A_1B_1)
\ge
F(\theta,h).
\]
\qed

\smallskip
\noindent \textbf{Remark.} (compare this with the remark after Theorem~\ref{thm:theorem_2}). It is not difficult to see that the function $F(\theta, h)$ is non-decreasing in both variables $\theta \in[0,\pi]$ and $h\ge 0$. Moreover, for any fixed $\theta\in (-\pi,\pi)$, the function $F(\theta, h)$ is strictly increasing in $h\ge 0$ and for any fixed $h\ge 0$, the function $F(\theta,h)$ is strictly increasing in $\theta\in[0,\pi]$. 

Note also that $F(\theta,0)=|\theta|$, $F(0,h)=2h$, and $F(\pi,h)=\max\{\pi, 2\sqrt{h^2+1}\}$.


\bigskip
\begin{thebibliography}{}

\bibitem{Bo_Roth_1990} O. Bottema and B. Roth,
Theoretical Kinematics,
Dover, 1990.

\bibitem{Dubovitskii1976}
A.~Ya.~Dubovitskii,
\newblock Solution of S.~Ulam's Problem on the Optimal Matching of Segments,
\newblock {\em Soviet Mathematics Doklady},
16:1373--1376, 1976.
Translated from
{\em Doklady Akademii Nauk SSSR},
236(1):17--20, 1977.

\bibitem{IckingRoteWelzlYap1993}
C.~Icking, G.~Rote, E.~Welzl, and C.-K.~Yap,
\newblock Shortest Paths for Line Segments,
\newblock {\em Algorithmica}, 10(2--4):182--200, 1993.
\newblock doi:10.1007/BF01891839.

\bibitem{Gurevich1977}
A.~B.~Gurevich,
\newblock On the Most Economical Motion of a Segment in the Plane,
\newblock {\em Mathematical Notes},
21(2):150--152, 1977.
Translated from
{\em Matematicheskie Zametki},
21(2):265--269, 1977.

\bibitem{Ulam1960}
S.~M.~Ulam,
\newblock {\em A Collection of Mathematical Problems},
\newblock Interscience Publishers, New York, 1960.
(Reprinted as {\em Problems of Modern Mathematics},
Science Editions, New York, 1964.)

\end{thebibliography}
\end{document}